\documentclass[11pt]{amsart}

\usepackage{amsmath,amsthm,epsfig,graphicx,subfigure}

\usepackage{amssymb}

\newcommand{\R}{{\mathbb R}}       
\newcommand{\N}{{\mathbb N}}       %

\newcommand{\HH}{{\mathcal H}}

\newcommand{\dist}{{\rm dist}}

\newcommand{\interior}[1]{{\stackrel{\mbox{\scriptsize$\circ$}}{#1}}}

\newcommand{\rf}[1]{{(\ref{#1})}}

\newcommand{\Rieszmu}{{R^s_{\mu}}}
\newcommand{\Riesznu}{{R^s_{\nu}}}

\newcommand{\ve}{{\varepsilon}}
\newcommand{\vv}{{\vspace{2mm}}}
\newcommand{\vvv}{{\vspace{3mm}}}
\newcommand{\wt}[1]{{\widetilde{#1}}}

\newtheorem{theorem}{Theorem}[section]
\newtheorem*{maintheorem*}{Theorem 1.1}
\newtheorem*{theorem*}{Theorem}
\newtheorem*{theoremb*}{Theorem B}
\newtheorem*{theoremc*}{Theorem C}
\newtheorem*{theoremd*}{Theorem D}
\newtheorem{lemma}[theorem]{Lemma}
\newtheorem{mlemma}[theorem]{Main Lemma}

\theoremstyle{definition}

\theoremstyle{remark}

\numberwithin{equation}{section}

\pdfoutput=1 

\begin{document}

\title[Non-existence of reflectionless measures]{Non-existence of reflectionless measures for the $s$-Riesz transform when $0<s<1$}

\author{Laura Prat and Xavier Tolsa}

\address{Laura Prat. Departament de Ma\-te\-m\`a\-ti\-ques, Universitat Aut\`onoma de Bar\-ce\-lo\-na, Catalonia}
\email{laurapb@mat.uab.cat}

\address{Xavier Tolsa. Instituci\'{o} Catalana de Recerca i Estudis Avan\c{c}ats (ICREA) and Departament de Ma\-te\-m\`a\-ti\-ques, Universitat Aut\`onoma de Bar\-ce\-lo\-na, Catalonia}
\email{xtolsa@mat.uab.cat}

\thanks{L.P. and X.T. were partially supported by the grants MTM-2013-44304-P (MICINN, Spain) and 
 2014-SGR-75 (Catalonia).
X.T.\ was also supported by 
the ERC grant 320501 of the European Research
Council (FP7/2007-2013).}


\begin{abstract}
A measure $\mu$ on $\R^d$ is called reflectionless for the $s$-Riesz transform if the singular integral
$R^s\mu(x)=\int \frac{y-x}{|y-x|^{s+1}}\,d\mu(y)$ is constant on the support of
$\mu$ in some weak sense and, moreover,  the operator defined by $R^s_\mu(f)=R^s(f\,\mu)$ is bounded in $L^2(\mu)$.
In this paper we show that the only reflectionless measure for the $s$-Riesz transform is the zero measure
when $0<s<1$.
\end{abstract}
\maketitle

\section{Introduction and Background}

Fix $d\in\N$ and for $0<s<d$, consider the signed vector valued Riesz kernels $$K^s(x)=\frac{x}{|x|^{1+s}},\;x\in\R^d,\;x\neq 0.$$
The $s$-Riesz transform of a real Radon measure $\mu$ is
$$R^s\mu(x)=\int K^s(y-x)d\mu(y),$$
whenever the integral makes sense. To avoid technical problems with the absolute convergence of the integral, one considers the 
truncated $s-$Riesz transform of $\mu$, which is defined as $$R_{\varepsilon}^s\mu(x)=\int_{|y-x|>\varepsilon}K^s(y-x)d\mu(y),\;\;x\in\R^d,\;\varepsilon>0.$$
Given a positive Radon measure and a function $f\in L^1(\mu)$, we consider the operators $\Rieszmu(f):=R^s(fd\mu)$ and
$R^s_{\mu,\ve}(f):=R^s_\ve(fd\mu)$. We say that $\Rieszmu$ is bounded in $L^2(\mu)$ if the truncated Riesz transforms $R^s_{\mu,\ve}$ are bounded in $L^2(\mu)$ uniformly in $\ve$, and we set
$$\|\Rieszmu\|_{L^2(\mu)\to L^2(\mu)}=\sup_{\ve>0}\|R^s_{\mu,\ve}\|_{L^2(\mu)\to L^2(\mu)}.$$

We are interested in the following conjecture: for non-integer $0<s<d$, the $s$-Riesz transform never defines a bounded operator on a set of positive and finite $s$-Hausdorff measure.
For $0<s<1$, this conjecture was solved in the affirmative by Prat \cite{laura1} (see also \cite{mpv}), and for 
$d-1<s<d$ by Eiderman, Nazarov and Volberg \cite{ENV} (relying on a previous result by Vihtil\"a
\cite{Vi}). In the case $0<s<1$, this can be proved by the symmetrization method, which was originally used
by Melnikov  
\cite{melnikov} in connection with the Cauchy kernel. For $0<s< 1$ and three different points $x_1,x_2,x_3\in\R^d$, consider the quantity
\begin{equation}\label{eqpermut}
 p_s(x_1,x_2,x_3)=\frac12\sum_\sigma\frac{x_{\sigma(2)}-x_{\sigma(1)}}{|x_{\sigma(2)}-x_{\sigma(1)}|^{1+s}}
 \cdot
 \frac{x_{\sigma(3)}-x_{\sigma(1)}}{|x_{\sigma(3)}-x_{\sigma(1)}|^{1+s}},
\end{equation}
where the sum is taken over the six permutations of the set $\{1,2,3\}$ and ``$\cdot$'' denotes the scalar product. In \cite[Lemma 4.2]{laura1} it was proved that 
\begin{equation*}
 p_s(x_1,x_2,x_3)\approx\frac1{\max\{|x_1-x_2|,|x_1-x_3|,|x_2-x_3|\}^{2s}}.
\end{equation*}
As shown in \cite[Theorem 4.4]{laura1},
it turns out that if $\mu$ is a finite positive Borel measure such that
$$\Theta^{s,*}(x,\mu):=\displaystyle{\underset{r\to 0}\limsup\frac{\mu(B(x,r))}{(2r)^s}}>0
\qquad\mbox{for $\mu$-a.e.\ $x\in\R^d$,}$$
 then
$$p_s(\mu)=\iiint p_s(x_1,x_2,x_3)=+\infty.$$
If, moreover, $\mu(B(x,r))\leq c\,r^s$ for all $x\in\R^d$ and all $r>0$ (which is a necessary condition for the $L^2(\mu)$ boundedness of $R_\mu^s$ if $\mu$ has no point masses), then we have 
$$\int|R^s_{\varepsilon}\mu(x)|^2d\mu(x) = \frac13
\iiint_{\begin{subarray}{l}
\mbox{   }\\\\|x_1-x_2|>\ve\\|x_1-x_3|>\ve\\|x_2-x_3|>\ve
\end{subarray}} p_s(x_1,x_2,x_3)d\mu(x_1)d\mu(x_2)d\mu(x_3)+O(\|\mu\|).$$
See \cite[Section 4.4]{laura1} for more details. As a consequence, for $0<s<1$, we deduce that
\begin{equation}\label{property}
\mbox{if $0<\Theta^{s,*}(x,\mu)<\infty$ $\mu$-a.e.\ and }\Rieszmu\;\mbox{is bounded in }L^2(\mu),\;\mbox{then $\mu$ is the zero measure.}
\end{equation}
In particular, there are no sets $E\subset \R^d$ with $0<\HH^s(E)<\infty$
such that for $\mu=\HH^s_{|E}$ the Riesz transform $R^s_\mu$ is bounded in $L^2(\mu)$.

In the case $d-1<s<d$, \rf{property} also holds, as well as the analogous consequence for $\mu=\HH^s_{|E}$, because of the aforementioned works of Eiderman, Nazarov and Volberg \cite{ENV} and Vihtil\"a \cite{Vi}.

On the other hand, \rf{property} remains open for non-integer values $s\in (1,d-1)$.
In \cite{jn} Jaye and Nazarov reduce this open problem to the problem of describing the 
reflectionless measures associated with $R^s_\mu$. One says that $\mu$ is reflectionless for the $s$-Riesz transform if $\mu$ has no point masses, $\Rieszmu$ is bounded in $L^2(\mu)$, and 
\begin{equation}\label{reflectionless}
\langle R^s\mu,\psi\rangle_\mu=0,
\end{equation}
for all Lipschitz continuous compactly supported functions $\psi$ such that $\int\psi d\mu=0$, where 
$R^s\mu$ is defined modulo constants as a weak limit of the functions $R^s_\ve\mu$ when $\ve\to0$.
Recall that a function $\psi$ is called Lipschitz continuous if 
$$\|\psi\|_{\tiny{\mbox{Lip}}}=\underset{x,y\in\R^d,x\neq y}{\sup}\frac{|\psi(x)-\psi(y)|}{|y-x|}<\infty.$$
In particular, it is shown in \cite{jn} that if, for a given $s\in(0,d)$, the only reflectionless measure for $R^s_\mu$ is the zero measure, 
then \eqref{property} holds. Further, in the same work the authors prove that in the case 
$s\in(d-1,d)$ there do not exist non-zero reflectionless measures for the $s$-Riesz transform, which 
yields a new proof of \rf{property} for this range of $s$.

The main result of this paper is the following:

\begin{theorem}\label{maintheorem}
 There are no non-trivial reflectionless measures for the $s$-dimensional Riesz transform if $0<s<1$.
\end{theorem}

Our arguments to prove this result combine the symmetrization method described above with
some variational techniques which have already appeared in other previous works, such as \cite{ENV} and \cite{ntov}.




\section{Preliminary results}

In what follows, we assume that $0<s<1$. Given a ball $B\subset\R^d$ of radius $r(B)$, we denote by $\theta^s_\mu(B)$ the average $s$-dimensional density of $\mu$ on $B$, that is 
$$\theta^s_\mu(B)=\frac{\mu(B)}{r(B)^s}.$$
Also we consider the ``Poisson'' type density
$${\mathcal P}(B)=\sum_{k\ge 0}\theta_\mu^s(2^kB)2^{-k}.$$
This can be considered as a smoothened version of $\theta^s_\mu(B)$.
Given two sets $E,F\subset\R^d$, we write
$$p_\mu(x,E,F)=\underset{E\times F}{\iint}p_s(x,y,z)d\mu(y)d\mu(z).$$
We say that $\mu$ has $s$-growth if there exists some constant $c$ such that
\begin{equation}\label{eqjka1}
\mu(B(x,r))\leq c\,r^s\qquad\mbox{for all $x\in\R^d$ and $r>0$.}
\end{equation}
It is well known that if $\mu$ has no point masses and $R_\mu^s$ is bounded in $L^2(\mu)$, then
$\mu$ has $s$-growth (see \cite[Proposition 1.4, p.56]{David-lnm}).
Further, as explained in the previous section, by \cite{laura1} $\mu$ satisfies
\begin{equation}\label{eqjka2}
\Theta^{s,*}(x,\mu)=0\qquad\mbox{for $\mu$-a.e.\ $x\in\R^d$.}
\end{equation}
As shown in \cite{Mattila-Verdera} (see also \cite[Chapter 8]{Tolsa-llibre}), then one deduces that for
all $f\in L^2(\mu)$,
\begin{equation}\label{eqjka3}
\lim_{\ve\to0} R_{\mu,\ve} f(x)\qquad \mbox{exists for $\mu$-a.e.\ $x\in\R^d$,}
\end{equation}
and it coincides $\mu$-a.e.\ with $R_{\mu} f(x)$, where $R_{\mu} f$ is the weak limit of $R_{\mu,\ve} f$ as $\ve\to0$.
In particular, this fact, as well as \rf{eqjka1}, \rf{eqjka2}  and \rf{eqjka3}, hold if $\mu$ is reflectionless for the $s$-Riesz transform.

We need now to recall the definition of balls with thin boundaries. Given $t>0$,
a ball $B(x,r)$ is said to have  $t$-thin boundary (or just thin boundary)
if
$$\mu\bigl(\{y\in B(x,2r):\dist(y,\partial B(x,r))\leq \lambda\,r\}\bigr) \leq t\,\lambda\,\mu(B(x,2r))$$
for all $\lambda>0$.
The following result is well known. For the proof (with cubes instead of balls) see Lemma 9.43 of \cite{Tolsa-llibre}, for
example.

\vv
\begin{lemma}\label{lemthin}
Let $\mu$ be a Radon measure on $\R^d$. Let $t$ be some constant big enough (depending only on $d$).
Let $B(x,r)\subset\R^d$ be any fixed ball. Then there exists $r'\in [r,2r]$ such that the ball $B(x,r')$ has $t$-thin boundary.
\end{lemma}
\vv

The next lemma will be also used later on.

\vv
\begin{lemma}\label{thin-estimates}
Let $\mu$ be a Radon measure on $\R^d$. Let $B$ be a ball with thin-boundary and $\|a\|_\infty<\infty$. Then

\begin{enumerate}
\item$\displaystyle{|\Rieszmu(a\chi_B)(x)|\le C\|a\|_\infty\theta_\mu^s(2B)\qquad \forall x\in2B\setminus B.}$
\vspace{2mm}

\item$\displaystyle{|\Rieszmu(a\chi_{2B\setminus B})(x)|\le C\|a\|_\infty\theta_\mu^s(2B)\qquad \forall x\in B.}$
\end{enumerate}
\end{lemma}
\vv

\begin{proof}
We will only show $(2)$, since the proof of $(1)$ follows the same reasoning.
Let $r$ denote the radius of our ball $B$. For $x\in B$ we have
\begin{equation*}
\begin{split}
 |\Rieszmu(a\chi_{2B\setminus B})(x)|&\le\|a\|_\infty\int_{2B\setminus B} \frac{d\mu(y)}{|y-x|^s}\\&\lesssim\|a\|_\infty\sum_{k\ge 1}\frac{\mu\left(B(x,2^{-k}r)\cap 2B\setminus B\right))}{(2^{-k}r)^s}
 \\&\lesssim\|a\|_\infty\sum_{k\ge 1}2^{-k(1-s)}\frac{\mu(2B)}{r^s}\lesssim\|a\|_\infty\theta_\mu^s(2B),
\end{split}
 \end{equation*}
since $\displaystyle{\mu\left(B(x,2^{-k}r)\cap 2B\setminus B\right)\lesssim 2^{-k}\mu(2B)}$ because of the fact that $$B(x,2^{-k}r)\cap 2B\setminus B\subset 2B\cap U_{2^{-k+1}r}(\partial B)$$
 and the thin-boundary property of $B$. Here $U_{\delta}(A)$ stands for the $\delta$-neighborhood of $A$.
\end{proof}

\vv
Below we denote by $\Rieszmu^*$ the adjoint of $\Rieszmu$. That is, if $R^s_{\mu,j}$ stands for the $j$-th component of
$R^s_\mu$, for $f=(f_1,\ldots,f_d)$ we have
$$\Rieszmu^* f(x) = - \sum_{j=1}^d R^s_{\mu,j} f_j(x).$$
We use define analogously the $\ve$-truncated operator $R^{s*}_{\mu,\ve}$.

\vv
\begin{lemma}\label{lemanou}
Let $\mu$ be a Radon measure on $\R^d$ with $s$-growth such that $\lim_{r\to0}\frac{\mu(B(x,r))}{(2r)^s} =0$ for
$\mu$-a.e.\ $x\in\R^d$. Then, for all bounded sets $E,F\subset \R^d$ and $\mu$-a.e.\ $x\in\R^d$,
\begin{equation}\label{identity}
\Rieszmu\chi_{E}(x)\cdot\Rieszmu\chi_F(x)+\Rieszmu^*\left((\Rieszmu\chi_{F})\chi_E\right)(x)
+ \Rieszmu^*\left((\Rieszmu\chi_E)\chi_{F}\right)(x)
=p_\mu(x,E,F).
\end{equation}
 In particular,
$$|\Rieszmu\chi_E(x)|^2+2\Rieszmu^*((\Rieszmu\chi_E)\chi_E)(x) = p_\mu(x,E,E).$$
\end{lemma}

\begin{proof}
 
 Fix $\ve>0$ and $x\in\R^d$. Consider the sets $$S_\ve=\{(y,z)\in E\times F:|y-x|>\ve,\;|z-x|>\ve,\;|z-y|>\ve\},$$ $$T^1_\ve=\{(y,z)\in E\times F:|y-x|>\ve,\;|z-y|>\ve,\;|z-x|\le\ve\}$$ and 
 $$T^2_\ve=\{(y,z)\in E\times F:|y-x|>\ve,\;|z-x|>\ve,\;|z-y|\le\ve\}.$$ Then, using
 \rf{eqpermut} (arguing as in \cite[Section 2]{MV}), 
 one can write
 \begin{equation}\label{eqdag33}
  \begin{split}
   R^s_{\mu,\ve}(\chi_{E})(x)&\cdot R^s_{\mu,\ve}(\chi_F)(x)+R^{s*}_{\mu,\ve}\left(R^s_{\mu,\ve}(\chi_{F})\chi_E\right)(x)+ R^{s*}_{\mu,\ve}\left(R^s_{\mu,\ve}(\chi_E)\chi_{F}\right)(x)
   \\   & =
  \underset{S_\ve}{\iint}p_s(x,y,z)d\mu(y)d\mu(z)+2\underset{T_\ve^1}{\iint}K(x-y)\cdot K(z-y)d\mu(y)d\mu(z)\\&\quad +\underset{T_\ve^2}{\iint}K(y-x)\cdot K(z-x)d\mu(y)d\mu(z)\\ &=
  p_{\mu,\ve}(x,E,F)+A_\ve(x)+B_\ve(x),
  \end{split}
 \end{equation}
 the last equality being a definition for $p_{\mu,\ve},\;A_\ve$ and $B_\ve$.

 Notice that as $\ve$ goes to zero, by \eqref{eqjka2} and \eqref{eqjka3},  the first term in the left hand side of \rf{eqdag33} satisfies
 \begin{equation}\label{eqtriv1}
 \lim_{\ve\to0}  R^s_{\mu,\ve}(\chi_{E})(x)\cdot R^s_{\mu,\ve}(\chi_F)(x) = R^s_{\mu}(\chi_{E})(x)\cdot R^s_{\mu}(\chi_F)(x)
\qquad \mbox{for $\mu$-a.e.\ $x\in\R^d$.}
\end{equation}
 On the other hand, it is also clear that 
 \begin{equation}\label{eqtriv2}
 \lim_{\ve\to0}  p_{\mu,\ve}(x,E,F) =  p_{\mu}(x,E,F) \qquad \mbox{for all $x\in\R^d$.}
\end{equation}

Concerning $A_\ve(x)$ since $|x-y|\approx |z-y|$ in the domain of integration, using the $s$-growth of $\mu$, we get
 \begin{equation}\label{eqtriv3}
 \begin{split}
 |A_{\ve}(x)|&\le\underset{\tiny{\begin{array}{l}|y-x|>\ve\\|z-x|\le\ve\\|z-y|>\ve\end{array}}}{\iint}\frac{d\mu(y)d\mu(z)}{|x-y|^s|z-y|^s}\lesssim\int_{|y-x|>\ve}\frac{\mu(B(x,\ve))d\mu(y)}{|x-y|^{2s}} \\\\&\lesssim\theta_\mu^s(B(x,\ve))\to 0 \quad \mbox{ as $\ve\to0$, for $\mu$-a.e.\ $x\in\R^d$.}\end{split}
 \end{equation}
 
Let us turn our attention to $B_\ve(x)$. Using that $|x-y|\approx |x-z|$ in the domain of integration, we derive
 \begin{equation*}
 \begin{split}
 |B_{\ve}(x)|&\lesssim\int_{y\in E:|y-x|>\ve}\frac{\mu(B(y,\ve))}{|y-x|^{2s}}d\mu(y).
 \end{split}
 \end{equation*}
 Integrating with respect to $x$ in $\R^d$, applying Fubini, and using the $s$-growth of $\mu$, we get
 \begin{align*}
\int |B_{\ve}(x)|\,d\mu(x) &\lesssim\iint_{y\in E:|y-x|>\ve}\frac{\mu(B(y,\ve))}{|y-x|^{2s}}d\mu(y) \,d\mu(x)\\
& = \int_{y\in E} \mu(B(y,\ve))  \int_{|y-x|>\ve} \frac1{|y-x|^{2s}}d\mu(x) \,d\mu(y)\\
&\lesssim \int_{y\in E} \frac{\mu(B(y,\ve))}{\ve^s} \,d\mu(y)\to 0\qquad\mbox{as $\ve\to0$,}
\end{align*}
 by the dominated convergence theorem. Hence we infer that there exists a sequence $\ve_k\to 0$ such that 
\begin{equation}\label{eqtriv4}
\lim_{k\to\infty} B_{\ve_k}(x)=0
\qquad \mbox{for $\mu$-a.e.\ $x\in\R^d$.}
\end{equation}

Finally we deal with the second and third terms on the left hand side of \rf{eqdag33}. 
We write
 $$f_\ve(y) = R^s_{\mu,\ve}(\chi_{F})(y)\chi_E(y)+ R^s_{\mu,\ve}(\chi_E)(y)\chi_{F}(y),$$
 so that $R^{s*}_{\mu,\ve}f_\ve(x)$ equals the sum of the last two terms on the left hand side of \rf{eqdag33}. Also, we set
$$f(y) = R^s_{\mu}(\chi_{F})(y)\chi_E(y)+ R^s_{\mu}(\chi_E)(y)\chi_{F}(y).$$ 
We claim that for $\mu$-a.e.\ $x\in\R^d$ there exists a subsequence $\ve_{k_j}\to0$ (depending on $x$) such that
\begin{equation}\label{eqclaim**}
\lim_{j\to\infty}R^{s*}_{\mu,\ve_{k_j}}f_{\ve_{k_j}}(x) = R^{s*}_{\mu}f(x).
\end{equation}
To prove this, notice that $R^s_{\mu,\ve_k}f(x)\to R^s_\mu f(x)$ 
as $k\to\infty$ for $\mu$-a.e.\ $x\in\R^d$ and thus for such $x$ we have
\begin{align*}
\liminf_{k\to\infty}\bigl|R^{s*}_{\mu,\ve_k}f_{\ve_k}(x) - R^{s*}_{\mu}f(x)\bigr|
& = \liminf_{k\to\infty}\bigl|R^{s*}_{\mu,\ve_k}f_{\ve_k}(x) - R^{s*}_{\mu}f(x) + R^{s*}_{\mu}f(x)  - R^{s*}_{\mu,\ve_k}f(x)\bigr|\\
& \leq \liminf_{k\to\infty}\bigl|R^{s*}_{\mu,\ve_k}(f_{\ve_k}- f)(x)\bigr|.
\end{align*}
Thus, for any $\lambda>0$, by Chebyshev's inequality, Fatou's lemma, and the $L^2(\mu)$ boundedness of $R_\mu^s$,
\begin{align*}
\mu\bigl(\bigl\{x\in\R^d\!: \liminf_{k\to\infty}\bigl|R^{s*}_{\mu,\ve_k}f_{\ve_k}(x) - R^{s*}_{\mu}&f(x)\bigr|>\!\lambda\bigr\} \bigr)\\ & \leq
\mu\bigl(\bigl\{x\in\R^d\!: \liminf_{k\to\infty}\bigl|R^{s*}_{\mu,\ve_k}(f_{\ve_k}\!- f)(x)\bigr|>\lambda\bigr\}\bigr)\\
& \leq \frac1{\lambda^2} \int \liminf_{k\to\infty}\bigl|R^{s*}_{\mu,\ve_k}(f_{\ve_k}- f)(x)\bigr|^2\,d\mu(x)\\
&\leq \frac1{\lambda^2} \liminf_{k\to\infty}\int \bigl|R^{s*}_{\mu,\ve_k}(f_{\ve_k}- f)(x)\bigr|^2\,d\mu(x)\\
&\leq \frac c{\lambda^2} \liminf_{k\to\infty}\bigl\|f_{\ve_k}- f\bigr\|_{L^2(\mu)}^2.
\end{align*}
Since the last limit vanishes, our claim \rf{eqclaim**} follows.

From the identity \rf{eqdag33}, for $\mu$-a.e.\ $x\in\R^d$, taking a suitable sequence $\ve_{k_j}\to0$, by
\rf{eqtriv1}, \rf{eqtriv2}, \rf{eqtriv3}, \rf{eqtriv4}, and \rf{eqclaim**}, we infer that the identity
\rf{identity} holds. 
\end{proof}

\section{The main Lemma}

\begin{mlemma}\label{mainlemma}
 Let $\mu$ be a reflectionless measure for the $s$-Riesz transform.  Let $B$ be a ball with thin boundary. Then, for $\mu$-a.e. $x\in B$,
 $$p_\mu(x,B,B)\lesssim{\mathcal P}(B)^2,$$
 with the implicit constant depending only on $s$ and $d$.
\end{mlemma}

To prove the lemma we will use a variational argument inspired by analogous techniques in \cite{ENV} and \cite{ntov}. The idea
is that since $\mu$ is a reflectionless measure, it minimizes some functional defined by the $L^2$ norm of some function.
Usually in potential theory, by variational techniques, one can show that the minimizers of some energy satisfy pointwise bounds. A similar phenomenon happens in our context. Let us also mention that the use of ideas from potential theory in the study of problems in connection with Riesz transforms is also present in other works such as \cite{Tolsa-sem} or \cite{Volberg}.

\begin{proof}
Let $\nu$ be a measure supported on $B$ of type $\nu=g\mu_{|B}$ and set 
\begin{equation*}
 F(\nu)=\int_B|\Rieszmu\chi_{B^c}+\Riesznu\chi_B-m_B^\nu(\Rieszmu\chi_{B^c})|^2d\nu,
\end{equation*}
where $\displaystyle{m_B^\nu(f)=\frac1{\nu(B)}\int_Bfd\nu}$. In what follows we write $m_B$ for $m_B^{\mu}$.

Notice that by antisymmetry and the reflectionless property of $\mu$, 
\begin{equation*}
 F(\mu\chi_B)=\int_B|R^s\mu-m_B(\Rieszmu\chi_{B^c})|^2d\mu=\int_B|R^s\mu-m_B(R^s\mu)|^2d\mu=0.
\end{equation*}
For a small $t\in\R$ and a ball $\Delta$, centered at $x\in B$, consider the measure
\begin{equation*}
 \nu_t=\mu\chi_B(1+t\chi_\Delta)=\mu\chi_B+t\mu\chi_{\Delta\cap B}
\end{equation*}
and set $g(t)=F(\nu_t)$. Then $g(t)\ge 0$ and $g'(0)=0$. 

Write $f_\nu= \Rieszmu\chi_{B^c}-m_B^\nu(\Rieszmu\chi_{B^c})$ and $f_{\nu_t}=f_t$, so that
\begin{equation*}
g(t)=\int_B|f_t+R_{\nu_t}^s\chi_B|^2d\nu_t=\int_B|R^s_{\nu_t}\chi_B|^2d\nu_t+2\int_B f_t\cdot R^s_{\nu_t}\chi_B\;d\nu_t+\int_B|f_t|^2d\nu_t.
\end{equation*}
Observe now that
$$f_t=\Rieszmu\chi_{B^c}-\frac1{\mu(B)+t\mu(\Delta\cap B)}\int_B\Rieszmu\chi_{B^c}\;d\mu-\frac t{\mu(B)+t\mu(\Delta\cap B)}\int_{\Delta\cap B}\Rieszmu\chi_{B^c}\;d\mu.$$
Notice that $\partial_t{f_t}_{|t=0}$ is constant. Indeed, we have
$$\partial_t{f_t}_{|t=0}=\frac{\mu(\Delta\cap B)}{\mu(B)^2}\int_B\Rieszmu\chi_{B^c}\;d\mu-\frac1{\mu(B)}\int_{\Delta\cap B}\Rieszmu\chi_{B^c}\;d\mu.$$
Hence, by using antisymmetry and the fact that $\partial_t{f_t}_{|t=0}$ is constant,
\begin{equation*}
\begin{split}
 0=g'(0)&=\int_{B\cap\Delta}|\Rieszmu\chi_B|^2d\mu+2\int_B\Rieszmu\chi_B\cdot\Rieszmu\chi_{B\cap\Delta}d\mu+2\int_{B\cap\Delta}f_0\cdot\Rieszmu\chi_Bd\mu\\&\quad+2\int_{B}f_0\cdot\Rieszmu\chi_{\Delta\cap B}d\mu+\int_{\Delta\cap B}|f_0|^2d\mu+2\int_Bf_0\cdot\partial_t{f_t}_{|t=0}d\mu\\
 &=\int_{B\cap\Delta}|\Rieszmu\chi_B|^2d\mu+2\int_{\Delta\cap B}\Rieszmu^*[(\Rieszmu\chi_B)\chi_B]d\mu+2\int_{B\cap\Delta}f_0\cdot\Rieszmu\chi_Bd\mu\\&\quad+2\int_{\Delta\cap B}\Rieszmu^*(f_0\chi_B)d\mu+\int_{\Delta\cap B}|f_0|^2d\mu+2\int_Bf_0\cdot\partial_t{f_t}_{|t=0}d\mu.
 \end{split}
\end{equation*}
Then, assuming that $x$ is a Lebesgue point for the functions 
$|\Rieszmu\chi_B|^2$, $\Rieszmu^*[(\Rieszmu\chi_B)\chi_B]$, $f_0\cdot\Rieszmu\chi_B$, $\Rieszmu^*(f_0\chi_B)$ and $|f_0|^2$, 
 and that $x\in \interior{B}$, we deduce
\begin{equation}\label{first}
\begin{split}
0=\lim_{r(\Delta)\to 0}\frac{g'(0)}{\mu(\Delta\cap B)}&=|\Rieszmu\chi_B(x)|^2+2\Rieszmu^*((\Rieszmu\chi_B)\chi_B)(x)\\&\quad+2f_0(x)\cdot\Rieszmu\chi_B(x)+
2\Rieszmu^*(f_0\chi_B)(x)+h(x),
\end{split}
\end{equation}
where $$h(x)=|f_0(x)|^2+\frac{2}{\mu(B)^2}\int_B\Rieszmu\chi_{B^c}d\mu\cdot\int_Bf_0d\mu-2\Rieszmu\chi_{B^c}(x)\cdot\frac1{\mu(B)}\int_Bf_0d\mu.$$

Write
\begin{equation*}
\begin{split}
f_0&=\left(\Rieszmu\chi_{2B\setminus B}-m_B(\Rieszmu\chi_{2B\setminus B})\right)+\left(\Rieszmu\chi_{(2B)^c}-m_B(\Rieszmu\chi_{(2B)^c})\right)=h_1+h_2,
\end{split}
\end{equation*}
the last identity being a definition for $h_1$ and $h_2$. Note that, for $w\in B$, by Lemma \ref{thin-estimates},  $|h_1(w)|\lesssim\theta_\mu^s(2B).$ Moreover by standard estimates (which can be obtained by splitting the domain of integration into annuli), we get
$$|h_2(w)|\le \frac1{\mu(B)}\int_B\int_{(2B)^c}|K(y-w)-K(y-z)|d\mu(y)d\mu(z)\lesssim{\mathcal P}(B).$$ 
Therefore, for all $w\in B$,
\begin{equation}\label{f0}
|f_0(w)|\lesssim \theta(2B)+{\mathcal P}(B)\lesssim{\mathcal P}(B).
\end{equation}
Hence,
\begin{equation}\label{eq3.3}
 h(x)=|f_0(x)|^2+2m_B(f_0)\cdot m_B(\Rieszmu\chi_{B^c}-\Rieszmu\chi_{B^c}(x))\lesssim{\mathcal P}(B)^2.
\end{equation}

Set
\begin{equation*}
A_1=|\Rieszmu\chi_B(x)|^2+2\Rieszmu^*((\Rieszmu\chi_B)\chi_B)(x)
\end{equation*}
and 
\begin{equation*}
 A_2=2f_0(x)\cdot\Rieszmu\chi_B(x)+2\Rieszmu^*(f_0\chi_B)(x).
\end{equation*}
In view of \eqref{first}, \eqref{f0}, and \eqref{eq3.3},
\begin{equation}\label{a+b}
|A_1+A_2|\lesssim{\mathcal P}(B)^2.
\end{equation}
Since $\theta_\mu^s(w)=0$ $\mu-$almost everywhere (see \cite[Theorem 4.4]{laura1}), by Lemma \ref{lemanou}, $A_1$ equals
 $p_\mu(x,B,B)$ $\mu$-a.e. on $B$. We deal now with $A_2$. Write
\begin{equation}
\label{third}
A_2=2\sum_{j=1}^2\left(h_j(x)\cdot\Rieszmu\chi_B(x)+\Rieszmu^*(h_j\chi_B)(x)\right).
\end{equation}
Notice that for any vectorial constant $k\in\R^d$, one has 
$k\cdot\Rieszmu\chi_B(x)+\Rieszmu^*(k\chi_B)(x)=0.$
Therefore for $k=m_Q(\Rieszmu\chi_{2B\setminus B})$ and $j=1$ in \eqref{third}, for $\mu$-a.e.\ $x\in B$ we obtain
\begin{equation*}
\begin{split}
 h_1(x)\cdot\Rieszmu\chi_B(x)+\Rieszmu^*(h_1\chi_B)(x)&=\Rieszmu\chi_{2B\setminus B}(x)\cdot\Rieszmu\chi_B(x)+\Rieszmu^*\left((\Rieszmu\chi_{2B\setminus B})\chi_B\right)(x)\\&=p_\mu(x,B,2B\setminus B)-\Rieszmu^*\left((\Rieszmu\chi_B)\chi_{2B\setminus B}\right)(x),
\end{split}
\end{equation*}
the last step due to \eqref{identity}.

We deal now with $j=2$ in \eqref{third}. Notice that for $y\in B$,
$$|h_2(x)-h_2(y)|=|\Rieszmu\chi_{(2B)^c}(x)-\Rieszmu\chi_{(2B)^c}(y)|\lesssim\int_{(2B)^c}\frac{|x-y|}{|z-y|^{s+1}}d\mu(z)\lesssim\frac{|x-y|}{r(B)}{\mathcal P}(B),$$
hence
\begin{equation*}
\begin{split}
h_2(x)\cdot\Rieszmu\chi_B(x)+\Rieszmu^*(h_2\chi_B)(x)&=\left|\int_B\frac{x-y}{|x-y|^{s+1}}h_2(x)d\mu(y)-\int_B\frac{x-y}{|x-y|^{s+1}}h_2(y)d\mu(y)\right|\\\\&\leq\int_B\frac{|h_2(x)-h_2(y)|}{|x-y|^s}d\mu(y)
\lesssim\frac{{\mathcal P}(B)}{r(B)}\int_B\frac{|x-y|}{|x-y|^s}d\mu(y)\\\\&\lesssim\frac{{\mathcal P}(B)}{r(B)^s}\,\mu(B)\lesssim{\mathcal P}(B)^2.
\end{split}
\end{equation*}
So we have shown that
$$|A_2 - 2\, p_\mu(x,B,2B\setminus B)|\leq 2\,|\Rieszmu^*\left((\Rieszmu\chi_B)\chi_{2B\setminus B}\right)(x)| + c\,{\mathcal P}(B)^2.$$
Therefore, from \eqref{a+b}, taking into account that $A_1= p_\mu(x,B,B)$ $\mu$-a.e.\ on $B$, we deduce
$$p_\mu(x,B,B)+2\,p_\mu(x,B,2B\setminus B)\lesssim{\mathcal P}(B)^2+|\Rieszmu^*\left((\Rieszmu\chi_B)\chi_{2B\setminus B}\right)(x)|\quad \mbox{for $\mu$-a.e.\ $x\in B$.}$$

By using Lemma \ref{thin-estimates} with $a=\Rieszmu(\chi_B)\chi_{2B\setminus B}$,  we have $\|a\|_\infty\lesssim\theta_\mu^s(2B)$, and hence we get that
\begin{equation*}
 |\Rieszmu^*\left((\Rieszmu\chi_B)\chi_{2B\setminus B}\right)(x)|\lesssim \theta_\mu(2B)^2,
\end{equation*}
and then immediatly we obtain 
$$p_\mu(x,B,B)+2p_\mu(x,B,2B\setminus B)\lesssim{\mathcal P}(B)^2 \quad \mbox{for $\mu$-a.e.\ $x\in B$},$$
which proves our lemma. 
\end{proof}

\section{Proof of the theorem}
\begin{lemma}\label{compdens}
 Let $0<s<1$ and let $\mu$ be a reflectionless measure for the $s$-Riesz transform. 
There exists some constant $\delta>0$ small enough (depending only on $s$ and $d$) such that if 
 $B_0\subset\R^d$ is a ball with $$\theta_\mu^s(B_0)\geq\sup_{x\in\R^d,r>0}\frac{\theta_\mu^s(B(x,r))}{2},$$ then every
ball $B\subset\R^d$ such that $B_0\subset\delta B$ satisfies $\displaystyle{\theta_\mu^s(B)\approx\theta_\mu^s(B_0)}$,
 with the implicit constant depending only on $s$ and $d$.
 \end{lemma}
 
Notice that in this lemma the ball $B$ can be arbitrarily big. We do {\em not} ask $r(B)\approx r(B_0)$. 
 
\begin{proof}
We assume first that $B_0\subset\frac 1 2 B$.  Let $B'$ be a ball with thin 
boundary such that $\frac 1 2 B\subset B'\subset B$. Then, by \cite{laura1} and Lemma \ref{mainlemma}, for $\mu$-a.e.\ $x\in B_0$,
\begin{equation*}
 \theta_\mu^s(B_0)^2\lesssim\iint_{B'\times B'}\frac{d\mu(y)d\mu(z)}{\max\{|x-y|,|x-z|,|y-z|\}^{2s}}\lesssim p_\mu(x,B',B')\lesssim{\mathcal P}(B')^2.
\end{equation*}
Therefore $$\theta_\mu^s(B_0)\lesssim{\mathcal P}(B')\lesssim{\mathcal P}(B).$$
We claim that this implies the existence of a ball $\wt B\supset B$, concentric with $B$, with comparable radius, that is $r(\wt B)\approx r(B)$, and such that 
$$\theta_\mu^s(\wt B)\geq \theta_\mu^s(B).$$
To see this, notice that 
\begin{align*}
\theta_\mu^s(B_0)\le{\mathcal P}(B)=\sum_k\theta_\mu^s(2^kB)2^{-k}& \le 
\max_{1\leq k\leq j}\theta_\mu^s(2^kB)+2\sum_{k\geq j+1}\theta_\mu^s(B_0)2^{-k}\\
&\leq 
\max_{1\leq k\leq j}\theta_\mu^s(2^kB)+
2^{-j+3}\theta_\mu^s(B_0),
\end{align*}
 which implies 
$$\theta_\mu^s(B_0)\lesssim \max_{1\leq k\leq j}\theta_\mu^s(2^kB)$$ 
for $j$ big enough and hence the claim. 
Thus, for every ball $B$ with $B_0\subset \delta B$ (for a sufficiently small $\delta$), $\theta_\mu^s(B)\gtrsim\theta_\mu^s(B_0)$. 

The converse inequality $\theta_\mu^s(B)\lesssim\theta_\mu^s(B_0)$ holds by the choice of $B_0$. So the lemma is proved.
\end{proof}

We now prove the main result of this paper:

\begin{theorem*}
 If $0<s<1$ there are no non-trivial reflectionless measures for the $s$-dimensional Riesz transform.
\end{theorem*}

\begin{proof}
Let $\mu$ be a non-trivial reflectionless measure for the $s$-dimensional Riesz transform. Let $B_0$ be a ball such that
$$\theta_\mu^s(B_0)\geq\underset{x\in \R^d,r>0}{\sup}\frac{\theta_\mu^s(B(x,r))}{2}.$$
Consider the sequence of balls $B_k=2^kB_0$.  By Lemma \ref{compdens}, there exists $k_0$ such that for $k\ge k_0$, $\theta_\mu^s(B_k)\approx\theta_\mu^s(B_0)$.  Since for $m\ge 0,$
$$\mu(B_{k+m})=r(B_{k+m})^s\,\theta_\mu^s(B_{k+m})\approx2^{ms}r(B_k)^s\theta_\mu^s(B_k)=2^{ms}\mu(B_k),$$
there exists $m\ge 0$ such that $\displaystyle{\mu(B_{k+m})\ge\frac{\mu(B_k)}{2}}.$
Consider the balls $B_{k_0},\;B_{k_0+m},\cdots,B_{k_0+Nm}$, and write $\wt B_j=B_{k_0+jm}$ for $0\le j\le N$. We can assume that $\wt B_N$ has thin boundary due to 
Lemma \ref{lemthin}. Then Lemma \ref{mainlemma} implies that for all $N$,
\begin{equation}\label{upperestimate}
p_\mu(x,\wt B_N,\wt B_N)\lesssim{\mathcal P}(\wt B_N)^2\approx\theta_\mu^s(B_0)^2.
\end{equation}
But  notice that
\begin{equation*}
\begin{split}
p_\mu(x,\wt B_N,\wt B_N)&\approx\underset{\wt B_N\times\wt B_N}{\iint}\frac{d\mu(y)d\mu(z)}{\max\{|x-y|,|x-z|,|y-z|\}^{2s}}\\&\gtrsim\sum_{j=1}^N\int_{\wt B_j}\int_{\wt B_j\setminus\wt B_{j-1}}\frac{d\mu(y)d\mu(z)}{r(\wt B_j)^{2s}}\\&=\sum_{j=1}^N\theta_\mu^s(\wt B_j)\frac{\mu(\wt B_j\setminus\wt B_{j-1})}{r(\wt B_j)^s}\\&\approx\sum_{j=1}^N\theta_\mu^s(\wt B_j)^2\approx N\theta_\mu^s(\wt B_0)^2\underset{N\rightarrow\infty}\longrightarrow \infty,
\end{split}
\end{equation*}
which contradicts \eqref{upperestimate}.
\end{proof}

\section{Non-existence of reflectionless measures for other Calder\'on-Zygmund kernels}

In this section we consider the vectorial Calder\'on-Zygmund kernels 
\begin{equation}\label{kn}
K^{s,n}(x)=\left(x_i^n/|x|^{n+s}\right)_{i= 1}^d,
\end{equation}
where $x=(x_1,\cdots,x_d)\in\R^d\setminus\{0\}$, $0<s<1$ and $n$ is an odd positive number. 

For $s=1$ and $d=2$ it was shown in \cite{cmpt} that for sets $E\subset\R^2$ with finite length the $L^2(\HH^1\lfloor{E})$-boundedness of
the singular integral associated with any kernel of the form $x_1^n/|x|^{n+1}$ (or $x_2^n/|x|^{n+1}$), $n$ odd, implies rectifiability of $E$, extending the theorem of David and L\'eger
\cite{davidleger}. In \cite{cp}, the kernels \eqref{kn} where considered to study the capacities related to them. It was shown that, on the plane, the capacity associated with the vectorial kernel \eqref{kn} is
comparable to the Riesz capacity $C_{\frac23(2-s),\frac33}$ of non-linear potential theory. Furthermore, in \cite{cp}, an extension to higher dimensions of the aforementioned result \cite{cmpt} was given.

For three different points $x_1,x_2,x_3\in\R^d$, consider the quantity
\begin{equation*}
 p_{s,n}(x_1,x_2,x_3)=\frac12\sum_\sigma K^{s,n}(x_{\sigma(2)}-x_{\sigma(1)})
 \cdot
 K^{s,n}(x_{\sigma(3)}-x_{\sigma(1)}),
\end{equation*}
where the sum is taken over the six permutations of the set $\{1,2,3\}$ and ``$\cdot$'' denotes the scalar product. In \cite[Corollary 3.2]{cp} it was proved that 
\begin{equation}\label{permn}
 p_{s,n}(x_1,x_2,x_3)\approx\frac1{\max\{|x_1-x_2|,|x_1-x_3|,|x_2-x_3|\}^{2s}}.
\end{equation}

We are now interested in studying the singular integral operator $T^{s,n}=(T^{s,n}_i)_{i=1}^d$, defined formally as $$T^{s,n}_i(\mu)(x)=\int K^{n,s}_i(x-y)f(y)d\mu(y),$$  
for a real Radon measure $\mu$; as well as the truncated operators $$T^{s,n}_{\ve}\mu(x)=\int_{|y-x|>\varepsilon}K^{s,n}(y-x)d\mu(y),\;\;x\in\R^d,\;\varepsilon>0.$$
For a function $f\in L^1(\mu)$, we consider the operators $T^{s,n}_{\mu}(f):=T^{s,n}(fd\mu)$ and
$T^{s,n}_{\mu,\ve}(f):=T^{s,n}_\ve(fd\mu)$.
We say that $T^{s,n}_\mu$ is bounded in $L^2(\mu)$ if the truncated $T^{s,n}_{\mu,\ve}$ are bounded in $L^2(\mu)$ uniformly in $\ve$, and we set
$$\|T^{s,n}_\mu\|_{L^2(\mu)\to L^2(\mu)}=\sup_{\ve>0}\|T^{s,n}_{\mu,\ve}\|_{L^2(\mu)\to L^2(\mu)}.$$

As for the $s$-Riesz transform, using \eqref{permn} and following \cite{laura1} one can deduce property \eqref{property} for the operators $T^{n,s}$. And hence, 
by analogous arguments as in the proof of Theorem \ref{maintheorem} one can prove:

\begin{theorem}
 There are no non-trivial reflectionless measures for $T^{s,n}$ if $0<s<1$.
\end{theorem}

By saying a measure $\mu$ is reflectionless for $T^{s,n}$ we mean that $\mu$ has no point masses, $T^{s,n}_{\mu}$ is bounded in $L^2(\mu)$, and 
\eqref{reflectionless} holds replacing $R^s$ by $T^{s,n}$. 

\vvv


\end{document}